\documentclass[12pt]{amsart}

\usepackage{geometry}                
\usepackage{graphicx}
\usepackage{subfig}
\usepackage{psfrag}
\usepackage{graphics}
\usepackage{amssymb}
\usepackage{color}
\newtheorem{thm}{Theorem}[section]
\newtheorem{cor}[thm]{Corollary}
\newtheorem{lem}[thm]{Lemma}
\newtheorem{claim}[thm]{Claim}

\theoremstyle{definition}

\theoremstyle{remark}
\newtheorem{rem}[thm]{Remark}

\numberwithin{equation}{section}

\makeatletter
\newcommand*{\rom}[1]{\expandafter\@slowromancap\romannumeral #1@}
\makeatother
\begin{document}

\title{Volume bounds of conic 2-spheres}
\author{Hao Fang}
\address{14 MacLean Hall, University of Iowa, Iowa City, IA, 52242}
\email{hao-fang@uiowa.edu}
\author{Mijia Lai}%
\address{800 Dongchuan RD, Shanghai Jiao Tong University, Shanghai, China,200240 }%
\email{laimijia@sjtu.edu.cn}%
\thanks{M.L.'s work is partially supported by Shanghai Sailing program No. 15YF1406200 and NSFC No. 11501360.}

\begin{abstract}
We obtain sharp volume bound for a conic 2-sphere in terms of its Gaussian curvature bound. We also give the geometric models realizing the extremal volume. In particular, when the curvature is bounded in absolute value by $1$, we compute the minimal volume of a conic sphere in the sense of Gromov. In order to apply the level set analysis and iso-perimetric inequality as in our previous works, we develop some new analytical tools to treat  regions with vanishing curvature.
\end{abstract}


\maketitle
\section{Introduction}
In this paper, we study the volume of a conic sphere when its Gaussian curvature is bounded.

Let us first introduce notations for conic surfaces. A metric $g$ on a closed surface is said to have conic singularity of order $\beta$ ($\beta>-1$) at $p$, if in a local holomorphic coordinate centered at $p$,
\[
g(z)=e^{2u} |z|^{2\beta} |dz|^2,
\]
where $u$ is continuous and $C^2$ away from $p$. The conic singularity is modeled on the Euclidean cone: $\mathbb{R}^2$ equipped with $|z|^{2\beta} |dz|^2$ is isometric to a flat cone of angle $2\pi(\beta+1)$ at the cone tip. A metric $g$ is said to represent the divisor $D=\sum_{i=1}^{n} \beta_i p_i$, if $g$ has conic singularities of order $\beta_i$ at $p_i$ and is smooth elsewhere.

The Gauss-Bonnet theorem for Riemannian surfaces with conic metrics becomes (c.f. ~\cite{Tr})
\begin{align} \label{GB}
\int_{M} K_g dv_g=2\pi(\chi(M)+|D|):=2\pi \chi(M,D),
\end{align}
where $|D|=\sum_{i=1}^{n} \beta_i$ is the degree of the divisor.

Troyanov~\cite{Tr} has systematically studied the prescribing curvature problem for conic surfaces. For $\chi(M,D)\leq 0$, Troyanov obtained several results parallel with those results of prescribing curvature problem on smooth surfaces. For $\chi(M, D)>0$, he further divided the problem into three cases:
\begin{description}
  \item[Subcritical] $|D|<2\min_{i} \beta_i$;
  \item[Critical] $|D|=2\min_{i} \beta_i$;
  \item[Supercritical] $|D|>2\min_{i} \beta_i$.
\end{description}
Among some positive results in the subcritical case, he identified the analytical difficulties in the critical and supercritical cases. Briefly speaking, the corresponding functionals in the variational approach lose compactness.

In the rest of the paper, we shall assume that $-1<\beta_1\leq \cdots\leq \beta_n\leq 0$. Our main result is a sharp volume bound for a conic sphere in terms of its Gaussian curvature bound. Such volume bound is significant in critical and supercritical cases.

\begin{thm}  [Main theorem] \label{T1} Let $(S^2, D, g)$ be a conic sphere, set $\alpha:=|D|-\min_{i}\beta_i$ and $\beta:=\min_{i}\beta_i$. Suppose $\beta\leq \alpha$ and $a\leq K_g\leq b$.

Then if $a=0$, we have
\begin{equation} \notag
{\rm Vol}(S^2,g)\geq V_{0,b}:=\frac{\pi (2+|D|)^{2}}{b(1+\beta)};
\end{equation}
if $a< 0$, we have
\begin{equation} \notag
{\rm Vol}(S^2,g)\geq V_{a,b}:=2\pi[\frac{\beta+1}{a}+\frac{\alpha+1}{b}-\frac{\sqrt{(b-a)(b(\beta+1)^{2}-a(\alpha+1)^{2})}}{ab}];
\end{equation}
and if $a>0$, we have
\begin{align} \notag
V_{\rm min}:=&2\pi[\frac{\beta+1}{a}+\frac{\alpha+1}{b}-\frac{\sqrt{(b-a)(b(\beta+1)^{2}-a(\alpha+1)^{2})}}{ab}]\leq \\ \notag
{\rm Vol}(S^2,g)&\leq 2\pi[\frac{\beta+1}{a}+\frac{\alpha+1}{b}+\frac{\sqrt{(b-a)(b(\beta+1)^{2}-a(\alpha+1)^{2})}}{ab}]:=V_{\rm max}.
\end{align}
\end{thm}

Meanwhile, we have the following geometric models realizing  extremal volume bounds in Theorem~\ref{T1}. Identify $(S^2,g)$ with $(\mathbb{C}^{*},g)$ via stereographic projection, then we have
\begin{thm}\label{T2} Let $(S^2, D, g)$ be a conic sphere, set $\alpha:=|D|-\min_{i}\beta_i$ and $\beta:=\beta_1$. Suppose $\beta\leq \alpha$ and $a\leq K_g\leq b$, then
\begin{enumerate}
  \item ${\rm Vol}(S^2,g)$ achieves $V_{a,b}$ if and only if $(S^2, D, g)$ is isometric to  $(\mathbb{C}^{*}, D=\alpha 0+\beta \infty, g_{\rm extr}=e^{2u_{a,b}}g_0)$;
  \item ${\rm Vol}(S^2,g)$ achieves $V_{0,b}$ if and only if $(S^2, D, g)$ is isometric to  $(\mathbb{C}^{*}, D=\alpha 0+\beta \infty, g_{\rm extr}=e^{2u_{0,b}}g_0)$;
  \item ${\rm Vol}(S^2, g)$ achieves $V_{\rm min}$ if and only if $(S^2, D, g)$ is isometric to $(\mathbb{C}^{*}, D=\alpha 0+\beta \infty, g_{\rm extr}=e^{2u_{\rm min}}g_0)$;
  \item ${\rm Vol}(S^2, g)$ achieves $V_{\rm max}$ if and only if $(S^2, D, g)$ is isometric to $(\mathbb{C}^{*}, D=\alpha 0+\beta \infty, g_{\rm extr}=e^{2u_{\rm max}}g_0)$.
\end{enumerate}
\end{thm}

The reader is referred to Sect. 3 for the detailed expressions of $u_{a,b}$, $u_{0,b}$, $u_{\rm min}$ and $u_{\rm max}$. All extremal models exhibit a similar geometrical feature: they are obtained by gluing together regions with constant curvature.

\begin{rem}
In the case of $n=1$, we have $\alpha=0$, our results still hold.
\end{rem}

Recall a football is a positive constant curvature sphere with two conic points of equal angle. In terms of the conformal factor on $\mathbb{C}^{*}$, let
\begin{align} \notag
e^{2u_{\rm football}}=\frac{4(1+\alpha)^2 |z|^{2\alpha}}{(1+|z|^{2(1+\alpha)})^2},
\end{align}
then $(\mathbb{C}^{*}, g=e^{2u_{\rm football}}g_0)$ is a football of $\rm curvature \equiv1$, with $D=\alpha0+\alpha\infty$. We denote by $S^2_{\alpha, a}$ a football of curvature $\equiv a$ with two conic points of order $\alpha$. When $\alpha=0$, we get the standard round sphere.

Similarly, let
\begin{align} \notag
e^{2u_{\rm hyp}}=\frac{4(1+\beta)^2 |z|^{2\beta}}{(1-|z|^{2(1+\beta)})^2}, \quad e^{2u_{\rm flat}}=|z|^{2\beta},
\end{align}
then $g=e^{2u_{\rm hyp}}g_0$ and $g=e^{2u_{\rm flat}}g_0$ defines locally a curvature $\equiv-1$ region and a curvature $\equiv0$ region, respectively. Both have a conic singularity at $z=0$ of cone angel $2\pi(1+\beta)$.

The geometric model (1) in Theorem~\ref{T2} is resulted by gluing a curvature $\equiv a$ region containing a conic point of order $\beta$ to a cap of a football $S^{2}_{\alpha, b}$. The geometric model (2) is obtained by attaching a curvature $\equiv 0$ region containing a conic point of order $\beta$ to a cap of a football $S^{2}_{\alpha, b}$. (3) and (4) are both constructed by gluing two caps of two footballs $S^2_{\alpha, b}$ and $S^{2}_{\beta, a}$. In fact, the resulted metrics from gluing are all $C^{1,1}$ across the gluing latitude.

The following illustration might give the reader a better idea for gluing of two footballs. 
\newpage 
\begin{figure}[htbp!]
\begin{center}
\includegraphics[height=90mm]{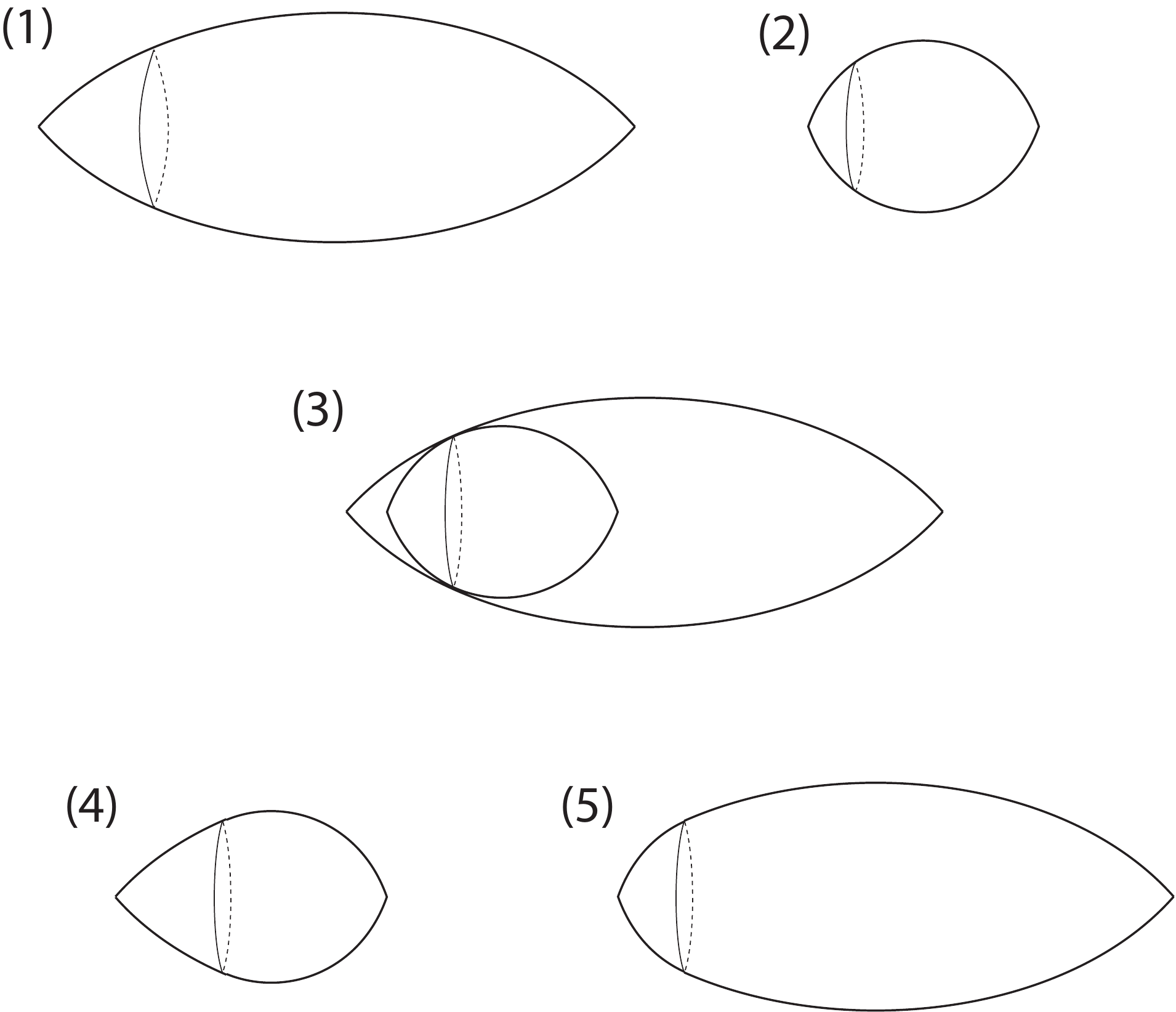}
\caption*{(1) football  $S^{2}_{\beta, a}$;\\
(2) football $S^2_{\alpha, b}$;\\
(3) two footballs placed together in a tangent position;\\
(4) glued football with angels $2\pi(1+\alpha)$ and $2\pi(1+\beta)$, $a\leq K\leq b$,  ${\rm V}_{min}$ model;\\
(5) glued football with angels $2\pi(1+\alpha)$ and $2\pi(1+\beta)$, $a\leq K\leq b$, ${\rm V}_{max}$ model.
}
\label{fig:balls}
\end{center}
\end{figure}

Following our previous works ~\cite{FL1, FL2}, we also find all geometric models above serve as Gromov-Hausdorff limits for any sequence of conic spheres $(S^2, D_l, g_l)$ when their volumes approach to the extremal bound.

\begin{thm}\label{T3} For any sequence of conic metrics $\{g_l\}$ on $(S^2, D_l=\sum_{i=1}^{n} \beta_i p^{l}_i)$ with $a\leq K_g\leq b$. Suppose $D_l$ is either critical with $n\geq3$ or $D_l$ is supercritical. We have
 \begin{enumerate}
                                      \item if ${\rm Vol}(S^2,g_l)\to V_{a,b}$, then $(S^2, D_l=\sum_{i=1}^{n} \beta_i p^{l}_i)$ converges in the Gromov-Hausdorff sense to $(\mathbb{C}^{*}, D=\alpha 0+\beta \infty, g_{\rm extr}=e^{2u_{a,b}}g_0)$, with $p^{l}_2, \cdots, p^{l}_n \to 0$, $p^{l}_1\to \infty$;
                                      \item if ${\rm Vol}(S^2,g_l)\to V_{0,b}$, then $(S^2, D_l=\sum_{i=1}^{n} \beta_i p^{l}_i)$ converges in the Gromov-Hausdorff sense to $(\mathbb{C}^{*}, D=\alpha 0+\beta \infty, g_{\rm extr}=e^{2u_{0,b}}g_0)$, with $p^{l}_2, \cdots, p^{l}_n \to 0$, $p^{l}_1\to \infty$;
                                      \item if ${\rm Vol}(S^2,g_l)\to V_{\rm min}$, then $(S^2, D_l=\sum_{i=1}^{n} \beta_i p^{l}_i)$ converges in the Gromov-Hausdorff sense to $(\mathbb{C}^{*}, D=\alpha 0+\beta \infty, g_{\rm extr}=e^{2u_{\rm min}}g_0)$, with $p^{l}_2, \cdots, p^{l}_n \to 0$, $p^{l}_1\to \infty$;
                                      \item if ${\rm Vol}(S^2,g_l)\to V_{\rm max}$, then $(S^2, D_l=\sum_{i=1}^{n} \beta_i p^{l}_i)$ converges in the Gromov-Hausdorff sense to $(\mathbb{C}^{*}, D=\alpha 0+\beta \infty, g_{\rm extr}=e^{2u_{\rm max}}g_0)$, with $p^{l}_2, \cdots, p^{l}_n \to 0$, $p^{l}_1\to \infty$.
 \end{enumerate}
\end{thm}

The primary motivation for our work comes from the minimal volume problem. The definition was first introduced by Gromov~\cite{G}. For a closed smooth manifold $M$, the minimal volume of $M$, ${\rm MinVol}(M)$, is defined to be the greatest lower bound of Vol$(M,g)$, where $g$ ranges over all complete Riemannian metrics on $M$ having sectional curvature bounded in absolute value by $1$, i.e.,
\begin{align}\notag
{\rm { MinVol}}(M):=\inf\{ {\rm Vol}(M,g)| |K_g|\leq 1\}.
\end{align}

We list a few results on minimal volume for smooth manifolds. For a smooth closed surface $M$, ${\rm MinVol}(M)=2\pi|\chi|$. There are many known examples of manifolds with ${\rm MinVol}(M)=0$. Such manifolds admit an $F$-structure (c.f. ~\cite{G,CG1,CG2}). If $M$ admits a finite-volume hyperbolic metric, then it is conjectured that this metric attains ${\rm MinVol}(M)$. For open manifolds, if $M$ is a topologically finite surface, not diffeomorphic to $\mathbb{R}^2$, then ${\rm MinVol}(M)=2\pi|\chi(M)|$ (c.f. ~\cite{B}). For $\mathbb{R}^n$, Bavard and Pansu~\cite{BP} proved that ${\rm MinVol}(\mathbb{R}^2)=2\pi(1+\sqrt{2})$. See ~\cite{B} for another proof. For $n\geq3$, Gromov~\cite{G} had shown that ${\rm MinVol}(\mathbb{R}^n)=0$. Mei-Wang-Xu~\cite{MWX} gave a detailed account of ${\rm MinVol}(\mathbb{R}^n)=0$ by gluing construction. In general, it is an interesting and difficult question to compute the minimal volume for a specific manifold.

For a closed surface $M$ and a fixed divisor $D$, we could define the minimal volume for $M$ among all metrics representing $D$ as follows:
\begin{align}\notag
{\rm MinVol}(M, D)=\inf\{ {\rm Vol}(M, g) | g \text{\ represents $D$ with $|K_g|\leq 1$}\}.
\end{align}

Under the curvature bound $|K_g|\leq 1$, it follows from the Gauss-Bonnet formula (\ref{GB}) that
\begin{align} \notag
2\pi|\chi(M,D)|=|\int_{M} K_g dv_g|\leq \int_{M} |K_g|dv_g={\rm Vol}(M,g).
\end{align}
The equality holds if and only if $K_g\equiv 1$ or $K_g\equiv -1$.

It has been shown, by the work of Troyanov~\cite{Tr} and McOwen~\cite{Mc}, that Uniformisation theorem holds for $\chi(M,D)\leq 0$. More precisely, if $\chi(M,D)<0$, then there admits a unique conformal conic metric $g$ representing $D$, with constant curvature $K_g\equiv-1$; if $\chi(M,D)=0$, then there admits flat conic metric $g$ representing $D$. Consequently, ${\rm MinVol}(M, D)=2\pi|\chi(M,D)|$ if $\chi(M,D)\leq0$.

In general, the Uniformisation theorem for conic surfaces with $\chi(M, D)>0$ does not hold. For example, a sphere with one conic singularity (a teardrop) does not support any metric with constant curvature. Note under the assumption $\chi(M,D)>0$ and $\beta_i \in (-1,0)$, $M$ must be a topological $2$-sphere. Through combined works of Troyanov, Chen-Li and Luo-Tian~\cite{Tr, CL,LT}, there is a complete characterization when the Uniformisation theorem holds for conic spheres:
a conic sphere $(S^2,D)$ admits a conic metric with positive constant curvature if and only if
\begin{enumerate}
 \item either $|D|<2\min_{i} \beta_i$,
  \item or $D=\beta_1 p+\beta_2 q$, $\beta_1=\beta_2$.
\end{enumerate}

Hence if $D$ is one of above two cases, ${\rm MinVol}(S^2, D)=2\pi\chi(M,D)$ as well. This leaves ${\rm MinVol}(S^2, D)$ unaccounted for when
$D$ is critical with more than $2$ conic points or $D$ is supercritical. The main theorem of this paper provides an answer:

\begin{thm} \label{T4} For a conic sphere $(S^2, D)$, with $D$ being either critical with $n\geq 3$ or supercritical, we have
\begin{align} \notag
{\rm MinVol}(S^2, D)=2\pi(\alpha-\beta+\sqrt{2(1+\alpha)^2+2(1+\beta)^2})=:V_{\alpha, \beta},
\end{align}
where  $\alpha=|D|-\min_{i}\beta_i$ and $\beta=\min_{i}\beta_i$.
\end{thm}

\begin{rem}
Based on the extremal models given by Theorem ~\ref{T2}, in our setting for $D$, ${\rm MinVol}(S^2, D)$ is only achieved if $D$ is supercritical with one or two conic points. The geometric model realizing ${\rm MinVol}(S^2, D)$ consists of two parts: one part contains a conic singularity of order $\alpha$ or a smooth part if $\alpha=0$ with curvature $\equiv 1$, the other part contains a conic singularity of order $\beta$ with curvature $\equiv-1$. This picture is somewhat similar to the geometric extremal realizing ${\rm MinVol}(\mathbb{R}^2)$, which is a spherical cap glued to the unbounded portion of the pseudosphere~\cite{B}.
\end{rem}

As a byproduct of Theorem~\ref{T1}, we get the following pinching estimate for conic spheres if its Gaussian curvature is bounded from below and above by two positive constants.
\begin{cor} \label{C1} For a supercritical conic sphere $(S^2, D, g)$, let $\alpha:=|D|-\min_{i}\beta_i$ and $\beta:=\beta_1$. Suppose the Gaussian curvature of $g$ satisfy that $0<a\leq K_g\leq b$, then
\begin{align} \label{Te1}
\frac{a}{b}\leq \frac{(\beta+1)^2}{(\alpha+1)^2}.
\end{align} The equality holds if and only if $(S^2, D, g)$ is isometric to $(\mathbb{C}^{*}, D=\alpha 0+\beta \infty, g_{extr}=e^{2u}g_0)$, where
\begin{align} \notag
e^{2u}=\begin{cases}
\frac{1}{b}\frac{4(1+\alpha)^2 |z|^{2\alpha}}{(1+|z|^{2(1+\alpha)})^2} \quad |z|\leq 1,\\
\frac{1}{a}\frac{4(1+\beta)^2}{(1+\frac{1}{|z|^{2+2\beta}})^2|z|^{4+2\beta}} \quad |z|\geq 1.
\end{cases}
\end{align}
\end{cor}

\begin{rem}
 The authors ~\cite{FL2} have obtained the estimate (\ref{Te1}) by considering the `least-pinched' metric problem on conic spheres. More precisely, one asks for the greatest upper bound of the pinching constant $\rho(g):=\frac{\min K_g}{\max K_g}$, where $g$ ranges over all conformal conic metrics representing $D$ with positive continuous curvature $K_g$. Such perspective was first taken on by Bartolucci~\cite{Ba} following the analysis of Chen-Lin~\cite{ChLi}. We recover our result~\cite{FL2} from the volume consideration.
\end{rem}

The basic idea of this paper follows closely to that of ~\cite{FL2}. Via stereographic projection, we study geometric quantities associated with corresponding conformal factors. The main tools are co-area formula and isoperimetric inequality. However, when the curvature lower bound is non-positive, extra effort is needed to take care of two subtle technical difficulties: First, the continuity of the distribution function of the conformal factor no longer holds; Second, absolutely continuity  is lost from similar consideration. We have developed some careful analysis to get around these obstacles.

The minimal volume question we consider in this paper can be dually interpreted as minimising the $L^{\infty}$-norm of the Gaussian curvature over all conic metrics with fixed volume. It shares a common feature with the curvature pinching problem considered in~\cite{FL2}:  both  are non-variational geometric extremal problems. Thus, it is expected that the geometric model for extremals may possibly lose smoothness. However, the use of the isoperimetric inequality forces the extremal to gain rotational symmetry so that we have clear geometric pictures and the Gromov-Hausdorff convergence in the corresponding moduli. We hope the study of these non-variational extremal problems will shed some light on similar questions. It is our intention to discuss corresponding topics for higher dimensional conic spheres with scalar curvature bound.

An outline of the paper is as follows. In Section 2, we prove the main theorem on the volume bound. In Section 3, we furnish the proofs of other results. Since the arguments are quite similar to those in ~\cite{FL2}, the presentation shall be brief.

{{\bf Acknowledgements}: The second author wishes to thank Prof. Jiaqiang Mei for raising the question on the minimal volume for conic spheres. }

\section{Proof of the main theorem}
In this section we follow the setup of~\cite{FL2} to estimate volume for conic spheres in terms of curvature bounds.

Let us first set up proper notations. Given a divisor $D= \sum_{i=1}^{n}\beta_{i}p_{i}$ on $S^2$ and $-1<\beta_{i}<0, i=1,\cdots,n$. By stereographic projection, we identify $S^{2}$ with $\mathbb{C}^{*}$. Let $z_i\in\mathbb{C}$ be the image of $p_{i}$ under the stereographic projection. Without loss of generality, we assume $z_{1}=\infty$. Let $g_{0}$ be the standard Euclidean metric on $\mathbb{C}$. Up to conformal transformations, we can assume the given conic metric $g$ is of the form $g=e^{2u}g_0$. Then the Gaussian curvature of $g$ satisfies
\begin{align} \label{GaussE}
\Delta u=-K_ge^{2u}, \quad \text{for $z\neq z_i$}.
\end{align}
The conic nature of $g$ is equivalent to the asymptotic behavior of $u$ near $z_{i}$:
\begin{itemize}
  \item  $u\sim\beta_{i}\ln|z-z_{i}|$
  as $z\to z_{i}, i>1$;
  \item  $u\sim-(\beta_{1}+2)\ln|z|$
  as $|z|\to z_{1}=\infty$.
\end{itemize}

Let $\alpha:=|D|-\beta_{1}=\sum_{i=2}^{n}\beta_{i}, \beta:=\beta_1$. We now assume that $(S^{2}, D)$ is supercritical or critical with more than 2 conic points. It is equivalent to $ \beta\leq\alpha\leq0$. Note we allow $\alpha=0$, which means there is only one conic point of order $\beta$. Also for simplicity, we denote $V=\rm{Vol}(S^{2},D,g)$.

In this section, we give a sharp estimate for $V$ in terms of curvature bounds. We would explore some geometric quantities associated with the conformal factor $u$ and apply co-area formula and isoperimetric inequality as in our previous works~\cite{FL1,FL2}. It turns out the volume $V$ is involved in an elementary inequality.

\begin{thm}\label{addt1}Let a supercritical or critical conic sphere $(S^2, D, g)$ be given as above and suppose it satisfy the curvature bound
$a\leq K_g\leq b$.

Then if $a=0$, we have
\begin{equation} \notag
V\geq V_{0,b}:=\frac{\pi (2+|D|)^{2}}{b(1+\beta)};
\end{equation}
if $a<0$, we have
\begin{equation} \notag
{\rm Vol}(S^2,g)\geq V_{a,b}:=2\pi[\frac{\beta+1}{a}+\frac{\alpha+1}{b}-\frac{\sqrt{(b-a)(b(\beta+1)^{2}-a(\alpha+1)^{2})}}{ab}];
\end{equation}
and if $a>0$, we have
\begin{align} \notag
V_{\rm min}:=&2\pi[\frac{\beta+1}{a}+\frac{\alpha+1}{b}-\frac{\sqrt{(b-a)(b(\beta+1)^{2}-a(\alpha+1)^{2})}}{ab}]\leq \\ \notag
{\rm Vol}(S^2,g)&\leq 2\pi[\frac{\beta+1}{a}+\frac{\alpha+1}{b}+\frac{\sqrt{(b-a)(b(\beta+1)^{2}-a(\alpha+1)^{2})}}{ab}]:=V_{\rm max}.
\end{align}
\end{thm}

\begin{rem}
It is a simple computation to show that when $a\to 0$, we have $$V_{a,b}\to V_{0,b}.$$
\end{rem}

We break of  proof of Theorem~\ref{addt1} into several steps.

\subsection{Level sets and related functions}

Define
\[
\Omega(t):=\{u>t\}\subset\mathbb{C}, \quad A(t):=\int_{\Omega_{t}}Ke^{2u},\quad  B(t):=|\Omega_{t}|, \quad s(t):=\int_{\Omega_{t}}e^{2u},
\]
where integrals are with respect to the Euclidean metric $g_{0}$ and $|\cdot|$ stands for the Lebesgue measure. Since $u(z)\to -\infty$ as $|z|\to \infty$, we know $B(t)$ is finite for any $t\in\mathbb{R}$.

The Gauss-Bonnet formula yields
\[
\int_{\mathbb{C}}Ke^{2u}=2\pi(2+|D|)=\lim_{t\to-\infty}A(t).
\]

In view of the asymptotic behavior of $u$ at singularities, we have $z_{i}\in\Omega(t), 2\leq i \leq n$, for any $t\in\mathbb{R}$. It then follows from the equation (\ref{GaussE}) that
\[
A(t)=\int_{\Omega(t)}Ke^{2u}=\int_{\Omega(t)}-\Delta u=\int_{\partial\Omega(t)}|\nabla u|+2\pi\alpha.
\]

\subsection{Critical Set}
It is clear from the definition that $s(t)$ is strictly decreasing with $s(-\infty)=V$. It also follows from definition that $B(t)$ and $s(t)$ are both continuous from right, with possible jump discontinuity at $t\in\mathbb{R}$ if and only if the level set $\{u=t\}$ has non-trivial Lebesgue measure. Define  $$s(t_{0}-):=\lim_{t\to t_0^{-}}s(t).$$
$$\mathcal{T}:=\{t\in\mathbb{R},\ s(t)\neq s(t-)\}.$$
Obviously, the set of discontinuous points $\mathcal T$ is at most countable. For future use, we denote by $C$ the set of critical points of $u$, i.e., $$C=\{z|\nabla u(z)=0\}.$$

\begin{rem} If $a>0$, we have $K\geq a>0$, the argument in~\cite{FL2} implies that $|C|=0$ (see also Lemma~\ref{add2} below). Thus according to the co-area formula (Lemma 2.3 of~\cite{BZ}), all  functions defined above is absolutely continuous with respect to $t$. Hence, $\mathcal{T}=\emptyset$, we can proceed as in~\cite{FL1,FL2}.
\end{rem}
For the general case, when $|C|>0$, functions $B$ and $s$ are not  necessarily absolutely continuous with respect to $t$. Instead, we use $s$ as our variable, and we shall prove that all relevant functions become absolutely continuous with respect to $s$.

We define some special subsets of the critical set $C$ and study their properties.  Define, for $d\in\mathbb{N}$,
\begin{align}\notag
\mathcal{N}(d):=\{ z|&\quad\nabla^{\gamma} u(z)=0 \quad \text{for all $|\gamma|=d$} \\ \notag
 and&\quad  \nabla^{\gamma}u (z)\neq0 \quad \text{for some $|\gamma|=d+1$}\},
\end{align}
where $\gamma$ the multi-index for mixed partial derivatives. For any $t\in\mathbb R$, define
\begin{align}\notag
\mathcal{N}(0)_{t}:=\{ z| u(z)=t,\ \ \ \nabla u(z)\neq 0\}.
\end{align}

For future use, we first prove the following
\begin{claim}\label{claim1}For any $t\in\mathcal R$, we have
\begin{align} \label{e2}
|\mathcal{N}(0)_{t}|=0 \quad \text{and} \quad |\mathcal{N}(d)|=0, \ \ \ d\geq 1.
\end{align}
\end{claim}

\begin{proof}$\forall z\in\mathcal{N}(0)_{t}$, without loss of generality, we may assume that $\frac{\partial u}{\partial x}(z)\neq0$. It then follows from the implicit function theorem that there exists $\rho>0$, such that $\{u=t\}\cap B_\rho(z)$ is the graph of some function $x=g(y)$. Clearly $\mathcal{N}(0)_{t}\cap B_{\rho}(z)\subset \{u=t_0\}\cap B_\rho(z)$, from which we infer that $|\mathcal{N}(0)_{t}|=0$.

Similarly, $\forall z\in\mathcal{N}(1)$, without loss of generality, we may assume that $\frac{\partial^2 u}{\partial x\partial y}(z)\neq0$. It follows from implicity function theorem that there exists $\rho'>0$, such that $\{u_x(z)=0\}\cap B_{\rho'}(z)$ is the graph of some function $y=h(x)$. Noticing that
$\mathcal{N}(1)\cap B_{\rho'}(z)\subset \{u_x(z)=0\}\cap B_{\rho'}(z)$, we get $|\mathcal{N}(1)|=0$ as well.
The proof for the general $d$ is similar, which we will omit here. We have thus finished the proof of Claim~\ref{claim1}.
\end{proof}

\begin{lem} \label{add2}
If $\Delta u\neq 0$, we have that $|C|=0$.
\end{lem}

\begin{proof}
$\Delta u\neq 0$ implies that $C={\mathcal N}(1)$. The conclusion is then obvious.
\end{proof}

\subsection{New variable and absolute continuity}
Now we want to define the `inverse function' for $s(t)$. Let ${\mathcal T}=\{\psi_n, \ n=1,2,\cdots\}$. $\{(s(\psi_n), s(\psi_n-))\}$ is then a family of disjoint open intervals in $[0,V]$. Set $\mathcal{S}=\cup_{n=1}^{\infty}(s(\psi_n), s(\psi_n-)$. Define
\begin{align}\notag
t(s)=\begin{cases} t,  \quad t\in\mathbb{R} \text{  such that $s=s(t)$, if $s\notin \mathcal{S}$};\\
                   \psi_n, \quad \text{if $s\in(s(\psi_n), s(\psi_n-)]$}.
                   \end{cases}
\end{align}

In other words, using vertical line segments to connect the 'gaps' (location of jump discontinuity) of the graph of $s(t)$, then viewing the graph from left to right, we get the graph of $t(s):[0,V]\to \mathbb{R}$. From the construction, we know $t(0)=\infty$ and $t(V)=-\infty$. We claim the following

\begin{lem}
With notations as above, $t(s)$ is locally Lipschitz. Hence, it is absolutely continuous. \label{t-ac}
\end{lem}
\begin{proof}
It follows from the definition that $t(s)$ is continuous and monotone non-increasing. Indeed, for $s_1>s_2$, we define $t_1:=t(s_1)\leq t_2:=t(s_2)$. If $t_{1}=t_{2}$, the claim holds trivially. Otherwise,  we have
$$s(t_{2})\leq s_{2}\leq s(t_{2}-)<s(t_{1})\leq s_{1}\leq s(t_{1}-)$$

Now by the co-area formula (see Lemma 2.3 in ~\cite{BZ}), we have
\begin{align}\label{e4}
s(t_1)-s(t_2-)=\int_{C\cap u^{-1}((t_1,t_2))} e^{2u} +\int_{t_1}^{t_2} \int_{u=\tau} \frac{e^{2u}}{|\nabla u|} d\mathcal{H}^1 d\tau:=I_s+II_s.
\end{align}
Hence
\[
s_1-s_2\geq s(t_{1})-s(t_{2}-)\geq\int_{t_1}^{t_2} \int_{u=\tau} \frac{e^{2u}}{|\nabla u|} d\mathcal{H}^1 d\tau.
\]
Equivalently
\[
|t(s_1)-t(s_2)|\leq (\int_{u=\tau} \frac{e^{2u}}{|\nabla u|} d\mathcal{H}^1)^{-1} |s_1-s_2|,
\] for some $\tau\in(t_1,t_2)$ by the mean value theorem. It then follows that $t(s)$ is locally Lipschitz.
\end{proof}

We now consider both quantities $A$ and $B$ as functions of $s$. For $A$, we simply take the composition as $\mathbb{A}(s):=A(t(s))$. It follows from the definition that $\mathbb{A}(0)=0$ and $\mathbb{A}(V)=\chi=2\pi(2+|D|)$.

\begin{lem}\label{add4}
With notations as above, $\mathbb{A}(s)$ is continuous on $[0,V]$.
\end{lem}

\begin{proof}
It suffices to prove that $A(t)$ is continuous. From the definition, it is clear that $A(t)$ is continuous from right, and
\[
\lim_{t\to t_0-} A(t)=\int_{\Omega_{t_0}} Ke^{2u}+\int_{u=t_0} Ke^{2u}.
\]

Thus $A(t)$ is continuous at $t_{0}$ provided $\int_{u=t_0} Ke^{2u}=0$. It then suffices to consider those $\psi\in\mathcal T$ for which $|\{u=\psi\}|>0$.

It follows from the definition that
\begin{align} \label{e1}
K(z)\equiv 0, \quad \forall z\in \{u=\psi \}\setminus (\mathcal{N}(0)_{\psi}\cup\mathcal{N}(1)).
\end{align}
Combining (\ref{e1}) and Claim~\ref{claim1}, we get $\int_{u=\psi} Ke^{2u}=0$, which finishes the proof.  \end{proof}

Furthermore, we prove the following
\begin{lem}
$\mathbb{A}(s)$ is Lipschitz. Furthermore, we have
\begin{align}\label{add31}
a\leq \mathbb{A}'(s)\leq b, \quad a.e. s\in[0,V].
\end{align}
\end{lem}
\begin{proof}
We proceed as in the proof of Lemma~\ref{t-ac}.  For $s_1>s_2$, we define $t_1:=t(s_1)\leq t_2:=t(s_2)$. If $t_{1}=t_{2}$, then $ t_{1}\in\mathcal T\neq \emptyset$, which by Lemma~\ref{add2} implies that $a<0$, thus (\ref{add31}) holds.

Otherwise, we have
\[s(t_{2})\leq s_{2}\leq s(t_{2}-)<s(t_{1})\leq s_{1}\leq s(t_{1}-).
\]By Lemma~\ref{add4} and (\ref{e1}),  we have
\begin{align} \label{e7}
{\mathbb{A}(s_1)-\mathbb{A}(s_2)}={\int_{\{t_2<u< t_1\}} Ke^{2u}}={\int_{\{t_2\leq u\leq t_1\}} Ke^{2u}}.
\end{align}
On the other hand, we have
\begin{align}\label{add32}
 \int_{t_{1}< u< t_{2}}e^{2u}\leq s(t_{1})-s(t_{2}-)\leq s_{1}-s_{2}\leq s(t_{1}-)-s(t_{2})=\int_{t_{1}\leq u\leq t_{2}}e^{2u}.
\end{align}
The conclusion then follows from  (\ref{e7}), (\ref{add32}) and the fact that $a\leq K\leq b$.
\end{proof}

Finally, we discuss the function $B(t)$. Define
\begin{align}\notag
\mathbb{B}(s):=\begin{cases} B(t(s)), \quad \text{ if $s\notin \mathcal{S}$};\\
                              \frac{(s-s(\psi_n))B(\psi_n-)+(s(\psi_n-)-s)B(\psi_n)}{s(\psi_n-)-s(\psi_n)}, \quad \text{if $s\in(s(\psi_n), s(\psi_n-)]$}.
                \end{cases}
\end{align}

It is clear that $\mathbb B$ is a monotone increasing function. Hence ${\mathbb B }'(s)$ exists almost everywhere.
We prove the following
\begin{lem}\label{add41}
For any $s_{1}>s_{2}$, we have
\begin{equation}\label{add45}
e^{-2t(s_{1})}\leq \frac{\mathbb{B}(s_1)-\mathbb{B}(s_2)}{s_1-s_2}\leq e^{-2t(s_{2})}.
\end{equation}
\end{lem}

\begin{proof}  Again, for $s_1>s_2$, we define $t_1:=t(s_1)\leq t_2:=t(s_2)$. If $t_{1}=t_{2}$, then $ t_{1}=\psi \in\mathcal T$, and $\mathbb B$ is defined linearly on $(s(t_{1}),s(t_{1}-))$. Thus, we have
\begin{equation}{\mathbb B}'(s)=  {\frac{B(\psi-)-B(\psi)}{s(\psi-)-s(\psi)}}={\frac{{\int_{u=\psi}1}}
{\int_{u=\psi}e^{2u}}}=e^{-2t_1},\label{44}
\end{equation}
which proves the statement.

Otherwise, we assume that
$$s(t_{2})\leq s_{2}\leq s(t_{2}-)<s(t_{1})\leq s_{1}\leq s(t_{1}-).$$
We then have
\[
\frac{\mathbb{B}(s(t_{1}))-\mathbb{B}(s(t_{2}-))}{s(t_{1})-s(t_{2}-)}=\frac{\int_{t_{1}< u< t_{2}}1}{\int_{t_{1}< u< t_{2}}e^{2u}}.
\]
Notice that similar to (\ref{44}), if $s_{2}< s(t_{2}-)$ or $s_{1}> s(t_{1})$, we have (\ref{add45}) in  $(s_{2},s(t_{2}-))$ and $(s_1, s(t_{1}-))$, respectively. A simple interpolation gives us the result.
\end{proof}

\begin{cor}
$\mathbb{B}(s)$ is locally Lipschitz. Furthermore, we have
\begin{align}\notag
\mathbb{B}'(s)=e^{-2t(s)}, \quad \forall s\in (0,V).
\end{align}
\end{cor}

We have thus established the regularity properties of $\mathbb{A}(s)$ and $\mathbb{B}(s)$ that are sufficient to the later estimate.

\subsection{Geometric inequality from iso-perimetric consideration}

Another direct consequence of the co-area formula is the following(c.f. Lemma 2.3 ~\cite{BZ})
\begin{align} \label{e6}
B'(t)\leq -\int_{u=t} \frac{1}{|\nabla u|} d\mathcal{H}^1, \quad \text{for a.e. $t\in \mathbb{R}$}.
\end{align}

By Sard's theorem, $\{u=t\}$ is a disjoint union of smooth closed curves, for $t$ almost everywhere. For such a $t$, the isoperimetric inequality and the H\"{o}lder's inequality yield that
\begin{equation}
4\pi B(t)\leq(\int_{\partial\Omega(t)}1)^{2}\leq\int_{\partial\Omega}|\nabla u|\int_{\partial\Omega(t)}\frac{1}{|\nabla u|}=\int_{\partial\Omega(t)}\frac{1}{|\nabla u|}(A(t)-2\pi\alpha).\label{51}
\end{equation}
Combining (\ref{e6}) and (\ref{51}), we have
\begin{align}
\label{E4} 4\pi B(t)\leq -B'(t) (A(t)-2\pi\alpha)=-{\mathbb B}' (s(t)) s'(t) (\mathbb{A}(s(t))-2\pi\alpha)\quad \text{for a.e. $t\in\mathbb{R}$}.
\end{align}

The key idea of the proof is the following estimate regarding the quantity $e^{2t(s)}\mathbb{B}(s)$, which by our previous analysis is absolutely continuous. Noticing that ${\mathbb B}'(s)=e^{-2t(s)}$, we thus get
\begin{align}
\frac{d}{ds}[e^{2t(s)}\mathbb{B}(s)]=&e^{2t(s)}2t'(s)\mathbb{B}(s)+e^{2t(s)}\mathbb{B}'(s)\notag\\
=& e^{2t(s)}2t'(s)\mathbb{B}(s)+1\notag\\
=& e^{2t(s)}2t'(s)\mathbb{B}(s)+1.\label{52}
\end{align}
Note that $t'(s)\leq 0$, we combine (\ref{E4}) and (\ref{52}) to get
\begin{align} \label{E10}
\frac{d}{ds}[e^{2t(s)}\mathbb{B}(s)]\geq1+\alpha-\frac{\mathbb{A}(s)}{2\pi}.
\end{align}
Note for $s\in \mathcal{S}$, we actually have $\mathbb{B}'(s)=t'(s)=0$, thus (\ref{E10}) holds trivially true.

Finally since
\[
\int_{\mathbb{R}^2} e^{2u}=2 \int_{-\infty}^{\infty} B(t)e^{2t} dt <\infty,
\] there exist sequences $t_n \to -\infty$ and $T_n\to \infty$, such that $e^{2t_n}B(t_n)\to 0$ and $e^{2T_n}B(T_n)\to 0$. Taking corresponding $s_n=s(t_n)$ and $S_n=s(T_n)$, we get by (\ref{E10}) that
\begin{align}\label{E6}
\int_{0}^{V}(1+\alpha-\frac{\mathbb{A}(s)}{2\pi})ds\leq \lim_{n\to\infty} \int_{S_n}^{s_n}\frac{d}{ds}[e^{2t(s)}\mathbb{B}(s)]ds=0.
\end{align}

\begin{rem} A slightly different version of (\ref{E10}) has first appeared in~\cite{FL2} where the curvature is positive. The absolute continuity is crucial to get (\ref{E6}) with weaker curvature bounds.
\end{rem}

\subsection{An extremal problem}Now we need to a technical lemma.
\begin{lem} \label{lemma}Let $f\in C^{1}([0,V])$ such that $f(0)=0$, $f(V)=\chi$, and $ a\leq f'\leq b$, then
\[
\int_0^{V} f(x)dx\leq \frac{-a}{2}V^2-\frac{(\chi-aV)^2}{2(b-a)}+V\chi.
\]
\end{lem}

\begin{proof}
Since $f(0)=0$, we have, due to integration by parts,
\[
 \int_{0}^{V}f(x) dx=\int_{0}^{V}f'(x)(V-x)dx.
\]

Let $\delta=\frac{\chi-aV}{b-a}$, then
\begin{align} \label{E5}
\int_{0}^{\delta}(b-f'(x))dx=\int_{\delta}^{V}(f'(x)-a)dx.
\end{align}
Notice that function $V-x$ is monotone decreasing and integrands of both sides of (\ref{E5}) are non-negative, we get by means of the mean value theorem that
\[
\int_{0}^{\delta}(b-f'(x))(V-x)dx\geq\int_{\delta}^{V}(f'(x)-a)(V-x)dx.
\]
A simple computation then leads to our conclusion.

It is obvious that when the equality  holds, $f$ has to be the following continuous function
\[g(x)=\begin{cases}
bx, & 0\leq s<\delta\\
a(x-V)+\chi, & \delta\leq s\leq V.
\end{cases}
\]
\end{proof}

\subsection{Proof of Theorem~\ref{addt1}}
Finally, we are ready to prove our main result in this section. Combining (\ref{E6}) and Lemma~\ref{lemma}, we get
\begin{align} \notag
2\pi(1+\alpha)V\leq -\frac{a}{2}V^2-\frac{(\chi-aV)^2}{2(b-a)}+V\chi,
\end{align}
where $\chi=2\pi(2+|D|)=2\pi(2+\alpha+\beta)$. Equivalently,
\begin{equation} \label{E13}
abV^{2}-4\pi(a(1+\alpha)+b(1+\beta))V+\chi^{2}\leq 0.
\end{equation}

Hence, when $a=0$, we get
\begin{equation} \notag
V\geq \frac{\pi (2+|D|)^{2}}{b(1+\beta)}.
\end{equation}
When $a\neq0$, we get
\begin{equation} \label{E7}
V\geq 2\pi[\frac{\beta+1}{a}+\frac{\alpha+1}{b}-\frac{\sqrt{(b-a)(b(\beta+1)^{2}-a(\alpha+1)^{2})}}{ab}].
\end{equation}
When $a>0$, we also get an upper bound for $V$:
\begin{align} \notag
V\leq 2\pi[\frac{\beta+1}{a}+\frac{\alpha+1}{b}+\frac{\sqrt{(b-a)(b(\beta+1)^{2}-a(\alpha+1)^{2})}}{ab}].
\end{align}
We have thus proven Theorem~\ref{addt1}.

When $a>0$, to ensure there exists at least one $V\in\mathbb{R}$ satisfies (\ref{E13}), we need the square root term in (\ref{E7}) is nonnegative. Thus we get the following necessary condition.

\begin{cor} \label{cor1}Assume the conditions of Theorem~\ref{addt1}. If $a>0$, we then have
 \[
\frac{a}{b}\leq \frac{(1+\beta)^2}{(1+\alpha)^2}.
\]
\end{cor}
Corollary~\ref{cor1} is first proved in~\cite{FL2}. Special cases when $n=1$ and $n=2$ have been previously proven by Chen-Lin \cite{ChLi} and Bartolucci \cite{Ba}, respectively. Here we have obtained an alternative argument.

\section{Concluding proofs}
In this section, we analyze the geometric models in various extremal situations. The idea is simply to trace cases for equalities in the proof of Theorem~\ref{addt1}. The equality case for isoperimetric inequality tells us level sets $\{u=t\}$ of the conformal factor in consideration are round circles for almost everywhere $t$. The equality case for the H\"{o}lder's inequality implies that these circles must be concentric. It then follows that the extremal models have to be rotationally symmetric and can allow at most two conic points. The corresponding quantities $A(t)$ and $B(t)$ uniquely determine the underlying geometry.

Even though the extremal models admit at most two conic points, following the arguments in ~\cite{FL1} and ~\cite{FL2}, we can construct conic metrics $g$ representing $D$ with proper curvature bound whose volumes approach to extremal volume bound, regardless of the number of conic points in $D$. This yields the answer for the minimal volume question. For the convergence part, we adopt the same idea in ~\cite{FL1} to study the isoperimetric deficit, from which the merging of conic points must occur.

\begin{proof}[Proof of Theorem ~\ref{T2}]
The proof hinges on analyzing the equality cases of Lemma~\ref{lemma}. When $V$ takes $V_{a,b}$, $V_{0,b}$, $V_{\rm min}$ and $V_{\rm max}$ respectively, we all have equality cases in Lemma~\ref{lemma}. It then follows
\[
\mathbb{A}(s)=\begin{cases}
bs, & 0\leq s<\delta\\
a(s-V)+2\pi(2+|D|), & \delta\leq s\leq V,
\end{cases}
\]
where $\delta=\frac{2\pi(2+|D|)-aV}{b-a}$.
After some calculations we find the corresponding conformal factors as follows:
\begin{itemize}
\item \begin{align} \label{Vab}
e^{2u_{a,b}}=\begin{cases}
\frac{1}{b}\frac{4(1+\alpha)^2 |z|^{2\alpha}}{(1+|z|^{2(1+\alpha)})^2} \quad |z|\leq r,\\
-\frac{1}{a}\frac{4(1+\beta)^2}{(1-\frac{1}{|z|^{2+2\beta}})^2|z|^{4+2\beta}} \quad |z|\geq r,
\end{cases}
\end{align}
where $r$ is uniquely determined by
\[
\int_{|z|\leq r}\frac{1}{b}\frac{4(1+\alpha)^2 |z|^{2\alpha}}{(1+|z|^{2(1+\alpha)})^2}=\frac{2\pi(2+|D|)-aV_{a,b}}{b-a}.
\]
\item \begin{align} \label{V0b}
e^{2u_{0,b}}=\begin{cases}
\frac{1}{b}\frac{4(1+\alpha)^2 |z|^{2\alpha}}{(1+|z|^{2(1+\alpha)})^2} \quad |z|\leq r,\\
\frac{1}{|z|^{4+2\beta}} \quad |z|\geq r,
\end{cases}
\end{align}
where $r$ is uniquely determined by
\[
\int_{|z|\leq r}\frac{1}{b}\frac{4(1+\alpha)^2 |z|^{2\alpha}}{(1+|z|^{2(1+\alpha)})^2}=\frac{2\pi(2+|D|)}{b}.
\]
\item \begin{align} \label{Vmin}
e^{2u_{\rm min}}=\begin{cases}
\frac{1}{b}\frac{4(1+\alpha)^2 |z|^{2\alpha}}{(1+|z|^{2(1+\alpha)})^2} \quad |z|\leq r,\\
\frac{1}{a}\frac{4(1+\beta)^2}{(1+\frac{1}{|z|^{2+2\beta}})^2|z|^{4+2\beta}} \quad |z|\geq r,
\end{cases}
\end{align}
where $r$ is uniquely determined by
\[
\int_{|z|\leq r}\frac{1}{b}\frac{4(1+\alpha)^2 |z|^{2\alpha}}{(1+|z|^{2(1+\alpha)})^2}=\frac{2\pi(2+|D|)-aV_{\rm min}}{b-a}.
\]
\item \begin{align} \label{Vmax}
e^{2u_{\rm max}}=\begin{cases}
\frac{1}{b}\frac{4(1+\alpha)^2 |z|^{2\alpha}}{(1+|z|^{2(1+\alpha)})^2} \quad |z|\leq r,\\
\frac{1}{a}\frac{4(1+\beta)^2}{(1+\frac{1}{|z|^{2+2\beta}})^2|z|^{4+2\beta}} \quad |z|\geq r,
\end{cases}
\end{align}
with $r$ determined uniquely by
\[
\int_{|z|\leq r}\frac{1}{b}\frac{4(1+\alpha)^2 |z|^{2\alpha}}{(1+|z|^{2(1+\alpha)})^2}=\frac{2\pi(2+|D|)-aV_{\rm max}}{b-a}.
\]
\end{itemize}
Note if $\alpha=\beta$, then $r=\infty$ in (\ref{Vab}), (\ref{V0b}), (\ref{Vmin}) and $r=0$ in (\ref{Vmax}).
\end{proof}

\begin{proof}[Proof of Theorem~\ref{T3}]
For the convergence, let
\[
D(t):=(\int_{\partial \Omega(t)} 1)^2-4\pi B(t)
\] be the isoperimetric deficit.
Taking this into account, (\ref{E10}) can be refined to
\begin{align} \label{E11}
\frac{d}{ds}[e^{2t(s)}\mathbb{B}(s)]\geq1+\alpha-\frac{\mathbb{A}(s)}{2\pi}+e^{2t(S)}\frac{D(t(s))}{2\pi}.
\end{align}
Then for a sequence of metrics $g_l$ with ${\rm Vol}(g_l)$ approaches to any of the extremal bounds $V_{a,b}$, $V_{0,b}$, $V_{\rm min}$ and $V_{\rm max}$, it is easy to see the corresponding $D_l(t)\to 0$. Thus one just follows the main argument in ~\cite{FL1} to conclude that after proper conformal gauge fixing, conformal factors $e^{2u_l}$ of $g_l$ converge to $e^{2u_{a,b}}$,$e^{2u_{0,b}}$,$e^{2u_{\rm min}}$ and $e^{2u_{\rm max}}$, respectively.  Moreover $p^{l}_2, \cdots p^{l}_n$ all merge to $0$.
\end{proof}

\begin{proof}[Proof of Theorem ~\ref{T4}] Given a conic metric $g$ on $(S^2,D)$, with $|K_g|\leq 1$. We would compare the volume lower bound given by Theorem~\ref{addt1}. Calculation shows that the minimum is the lower bound in Theorem~\ref{addt1} when taking $a=-1$ and $b=1$. By Lemma~\ref{lemma}, we infer
\[\mathbb{A}(s)=\begin{cases}
s, & 0\leq s<\delta\\
-(s-V)+\chi, & \delta\leq s\leq V,
\end{cases}
\] where $\delta=\frac{V+\chi}{2}$.
Then according to the computation above, the conformal factor is
\begin{align} \notag
e^{2u_{-1,1}}=\begin{cases}
\frac{4(1+\alpha)^2 |z|^{2\alpha}}{(1+|z|^{2(1+\alpha)})^2} \quad |z|\leq r,\\
\frac{4(1+\beta)^2}{(1-\frac{1}{|z|^{2+2\beta}})^2|z|^{4+2\beta}} \quad |z|\geq r,
\end{cases}
\end{align}
where $r$ is uniquely determined by
\[
\int_{|z|\leq r}\frac{4(1+\alpha)^2 |z|^{2\alpha}}{(1+|z|^{2(1+\alpha)})^2}=\frac{V_{\alpha, \beta}+2\pi(2+|D|)}{2}.
\]

Now we can follow the strategy of the proof of Theorem 3.1 in~\cite{FL2} to construct a sequence of approximated conformal factors $e^{2u_i}$ represents $D$, with the required curvature bound and ${\rm Vol}(g_i=e^{2u_i}g_0) \to V_{\alpha, \beta}$, which leads to ${\rm MinVol}(S^2, D)=V_{\alpha, \beta}$. It is important to note that if there are more than 3 conic points in $D$, all but one are placed in the region of positive curvature. Thus the construction in~\cite{FL2} work also in this setting.  We have thus finished the proof of Theorem~\ref{T4}.
\end{proof}

\begin{proof}[Proof of Corollary~\ref{C1}] The pinching estimate follows from Corollary~\ref{cor1}. If
\[
\frac{a}{b}=\frac{(1+\beta)^2}{(1+\alpha)^2},
\]
then there is only one $V\in\mathbb{R}$ satisfying the inequality
\[
abV^{2}-4\pi(a(1+\alpha)+b(1+\beta))V+\chi^{2}\leq 0.
\]
Hence $V=V_{\rm min}=V_{\rm max}$. It is also easy to see that $r=1$ in both (\ref{Vmin}) and (\ref{Vmax}). Geometrically, $\frac{a}{b}=\frac{(1+\beta)^2}{(1+\alpha)^2}$ implies that the equators of two footballs $S^2_{\beta, a}$ and $S^2_{\alpha, b}$ have equal length. The conic sphere in this case is isometric to the gluing two halves of above footballs along their equators.
\end{proof}

\end{document}